\documentclass[a4paper,12pt,reqno]{amsart}

\usepackage[applemac]{inputenc}
\usepackage[T1]{fontenc}

\usepackage{amsrefs, amssymb, microtype}
\usepackage{booktabs}  

\newcounter{notecounter}

\def\Z{\mathbf{Z}}
\def\F{\mathbf{F}}

\DeclareMathOperator{\PSL}{PSL}
\DeclareMathOperator{\GL}{GL}
\DeclareMathOperator{\SL}{SL}
\DeclareMathOperator{\Aut}{Aut}

\newtheorem{Th}{Theorem}

\newtheorem{Co}[Th]{Corollary}

\newtheorem*{WGT}{Weierstrass Gap Theorem}

\theoremstyle{definition}
\newtheorem{DEF}{Definition}
\newtheorem*{notes}{Notes}

\title{Transitivity on Weierstrass points }
\author{Zoë Laing \and David Singerman}
\date{\today}

\begin{document}

\maketitle

\thispagestyle{empty}

\section{Introduction} 

An automorphism of a Riemann surface will preserve its set of
Weierstrass points.  In this paper, we search for Riemann surfaces
whose automorphism groups act transitively on the Weierstrass points.
One well-known example is Klein's quartic, which is known to have 24
Weierstrass points permuted transitively by it's automorphism group,
$\PSL(2,7)$ of order $168$.  An investigation of when Hurwitz groups act
transitively has been made by Magaard and Völklein
\cite{magaard-06}.  After a section on the preliminaries, we examine
the transitivity property on several classes of surfaces.  The easiest
case is when the surface is hyperelliptic, and we find all
hyperelliptic surfaces with the transitivity property (there are
infinitely many of them).  We then consider surfaces with automorphism
group $\PSL(2, q)$, Weierstrass points of weight 1, and other classes
of Riemann surfaces, ending with Fermat curves.

Basically, we find that the transitivity property property seems quite rare and that the surfaces we have found with this property are interesting for other reasons too.
\section{Preliminaries}
\label{sec:prelim}

\begin{WGT}[\cite{farkas-80}]
    Let $X$ be a compact Riemann surface of genus $g$.  Then for each
    point $p\in X$ there are precisely $g$ integers
    $1=\gamma_1<\gamma_2<\ldots<\gamma_g<2g$ such that there is no
    meromorphic function on $X$ whose only pole is one of order
    $\gamma_j$ at $p$ and which is analytic elsewhere.
\end{WGT}

The integers $\gamma_1,\ldots,\gamma_g$ are called the gaps at $p$.
The complement of the gaps at $p$ in the natural numbers are called
the non-gaps at $p$.  Thus $\alpha$ is a non-gap at $p$ if there is a
meromorphic function on $X$ that has a pole of order $\alpha$ at $p$
and is holomorphic on $X\setminus \{ p \}$.  If $f$ has a pole of order
$\alpha$ at $p$, and $g$ has a pole of order $\beta$ at $p$ then $fg$
has a pole of order $\alpha+\beta$ at $p$ so that the non-gaps at $p$
form a semi-group under addition.

\begin{DEF}
    The \emph{weight} of $p$, denoted by $w_p$, is given by  
    \begin{equation}
	\label{eq:2}
	w_p=\sum_ {i=1}^ g(\gamma_i-i).
    \end{equation}
\end{DEF}

\begin{DEF} 
    $p$ is called a Weierstrass point if its weight is non-zero.
    Alternatively, $p$ is a Weierstrass point if there is a
    meromorphic function on $X$ with a pole of order $\leq g$ at $p$
    and analytic elsewhere.  (Thus $\alpha_1\leq g$ or $\gamma_g > g$
    at $p$.)
\end{DEF}

\begin{Th} \cite{farkas-80}
    \label{thrm:farkas1}
    \begin{equation}
	\label{eq:farkas}
	\sum_{p\in X }{w_p}=g^3-g
    \end{equation}
\end{Th}

\begin{Th} \cite{farkas-80}
    \label{thrm:farkas2}
    The number $|W|$ of Weierstrass points on $X$(of genus $g\ge 2$)
    obeys 
    \begin{equation*}
	2g+2\leq |W|\leq g^3-g.
    \end{equation*}
    Furthermore, $X$ has $2g+2$
    Weirstrass points if and only if $X$ is hyperelliptic (a two
    sheeted covering of the sphere) and the branch points are the
    Weierstrass points, (this follows from
    section~\ref{sec:hyp_surfaces}).  Generically, a Riemann surface
    of genus $g$ will have $g^3-g$ Weierstrass points (see
    Rauch~\cite{rauch-59}).
\end{Th}

A useful method of showing that $p$ is a Weierstrass points is
the following theorem whose proof can be found in
\cite{farkas-80} as Theorem~\textbf{V.1.7.} 

\begin{Th}
    \label{thrm:shoeneberg}
    (Schoeneberg)
    Let $h$ be an automorphism of a compact Riemann surface of genus
    $g\geq 2$.  If $h$ fixes more than $4$ points then these fixed
    points are Weierstrass points.
\end{Th}

This is equivalent to a theorem first stated by Schoeneberg in \cite{schoeneberg-51}.  In this form it is due to Lewittes \cite{lewittes-63}.

A compact Riemann surface $X$ of genus $g\geq 2$ can be represented as
$\mathcal U/K$ where $K$ is a torsion free Fuchsian group (a surface
group).  The full group of automorphisms of $X$ is then $N(K)/K$, and
it is well known that $N(K)$ is a Fuchsian group, and so a group $G$
of automorphisms of $X$ has the form $\Gamma/K$, where $\Gamma$ is a
cocompact Fuchsian group.  We call $\Gamma$ the lift of $G$.

Because of Schoeneberg's theorem we want to know the number of fixed
points of an automorphism.  For this we need a Theorem of
Macbeath~\cite{macbeath-73}.  Rather than stating this theorem in full
generality we will give a few consequences, also pointed out
in~\cite{macbeath-73}.  The following result is useful

\begin{Th}
    \label{thrm:4}
    Let $G$ be a cyclic group of automorphisms of $X$ of order $n$ and
    represent $G$ as $\Gamma/K$ as above. Suppose that $\Gamma$ has
    periods $m_1, \dots, m_r$.  If $t\in G$ has order $d$ then the
    number $F(t)$ of fixed points of $t$ is 
    \begin{equation*}
	F(t)= n\sum_{d|m_i} \frac{1}{m_i}.
    \end{equation*}
\end{Th}

In section~\ref{sec:surfaces_with_Aut_PSL2q} we will use automorphism
groups of the form $\PSL(2,q)$, where $q=p^n$, $p$ a prime.

\begin{Th}
    \label{thrm:5}
    Let $\Aut X \cong \PSL(2, q)$, $q$ odd and again represent $\Aut
    X=\Gamma/K$ as above.  If $t\in \Aut X$ has order $d$ and $q$ is
    odd then
    \begin{equation*}
	F(t)=
	\begin{cases}
	    \displaystyle (q-1)\sum_{d \mid m_i} \frac{1}{m_i}
	    & 
	    \text{if $d\mid\frac{q-1}{2}$ }
	    \\
	    \displaystyle (q+1)\sum_{d \mid m_i} \frac{1}{m_i}
	    &
	    \text{if $d\mid\frac{q+1}{2}$.}
	    \\
	    \displaystyle\frac {(n, 2)}{2}p^{n-1}(p-1)\sum_{m_i=p}
	    1
	    &
	    \text {if $d=p$}.\end{cases}
	\end{equation*}

	If $q$ is even then
	\begin{equation*}
	    F(t)=
	    \begin{cases}
		\displaystyle 2(q-1)\sum_{d\mid m_i} \frac{1}{m_i}
		&
		\text{if $d\mid {q-1}$ }
		\\
		\displaystyle 2(q+1)\sum_{d\mid m_i} \frac{1}{m_i}
		&
		\text{if $d\mid{q+1}$}
		\\
		\displaystyle 2^{n-1}\sum_{m_i=2}1
		&
		\text {if $d=2$}.
	    \end{cases}
	\end{equation*}
\end{Th}

\section{Maps on surfaces} 

It is now useful to introduce maps on surfaces.  We think of the word
\emph{map} in its geographic sense.  Basically we just decompose the
surface into simply-connected regions.

A formal theory was given in~\cite{jones-78}, although the origins of
the theory go much further back.  A map on a surface $X$ is an
embedding of a graph $\mathcal G$ into $X$ such that the components of
the complement of $\mathcal M$, which are called the faces of
$\mathcal M$, are simply-connected.  The detailed theory is set out in
\cite{jones-78} but for briefer versions see \cite{singerman-97} or
\cite{jones-96}.  There it is shown that if a surface carries a map,
then it has the structure of a Riemann surface.  A \emph{dart} of a
map $\mathcal M$ is a directed edge and the automorphism group of
the map is also an automorphism group of the associated Riemann
surface.

\begin{DEF}
    A map is \emph{regular} if the automorphism group of the map
    acts transitively on the darts of the map.
\end{DEF}

The theory of regular maps goes back to the end of the 19th century;
for a survey of the classical work see \cite{coxeter-80}.

\begin{DEF}(\cite {singerman-97})
    A \emph{geometric point} of a map is a vertex, face, centre or
    edge-centre.
\end{DEF}

\begin{DEF}
    A Riemann surface that contains a regular map is called
    \emph{Platonic}.
\end{DEF}

In~\cite{singerman-97} we investigated the question of when the geometric
points are Weierstrass points.

Platonic surfaces are those that are of the form $\mathcal{U}/K$
where $K$ is a torsion-free normal subgroup of a Fuchsian triangle
group, with one period equal to $2$, so that the automorphism group of a Platonic surface is an image
of a $(2,m,n)$ triangle group. Two very famous Platonic surfaces
are
\begin{enumerate}  
    \item Klein's Riemann surface of genus $3$.  This surface has
    automorphism group $\PSL(2, 7)$ of order $168$.  It corresponds to
    a map of type $\{3, 7\}$.
    
    \item Bring's surface of genus $4$.  This surface has automorphism
    group $S_5$ of order $120$.  It corresponds to a map of type $\{5,
    4\}$ which gives the small stellated dodecahedron.  For Bring's
    surface, see~\cite{weber-05}.
\end{enumerate}
These surfaces will appear again in 
section~\ref{sec:necessary_condition}.

Since Grothendieck's esquisse d'un programme, maps are known as clean
dessins d'enfants.  Grothendieck pointed out that Belyi's Theorem
implies that the Riemann surface underlying a map corresponds to an algebraic curve defined over the field of algebraic
numbers.

\section{Hyperelliptic surfaces}
\label{sec:hyp_surfaces}

Hyperelliptic surfaces are Riemann surfaces that are two-sheeted
coverings of the Riemann sphere.  They admit an automorphism $J$,
called the hyperelliptic involution, that interchanges the two sheets
and which is central in the automorphism group of the surface.  By the
Riemann--Hurwitz formula there are $2g+2$ branch points on the surface
which are fixed by $J$.  If the genus of the surface is greater than
$1$ then these branch points are the Weierstrass points on the surface
(see~\cite{farkas-80}*{\textbf{111.7}}), or just use Schoeneberg's
Theorem.  We want to find the cases where the automorphism group of a
hyperelliptic surface acts transitively on the $2g+2$ Weierstrass
points.  The group then acts transitively on the projections of these
Weierstrass points on the sphere.  We study the cases where these
points are vertices, face centres or edge centres of regular maps on
the sphere.  These are the Platonic solids.  A study of branched
coverings of the Platonic solids has been made by Jones and
Surowski~\cite{jones-00}, and we use their results.  (Except they
considered branching over face-centres, whereas we consider branching
over vertices.)  First, some notation.  If $\mathcal M$ is a platonic
solid (or any map on a surface), we let $v$, $e$, and $f$ denote the
number of vertices, edges and faces of $\mathcal M$.  We let $m$ be
the valency of a vertex, and $n$ the valency of a face, that is the
number of edges of the face.  The type of the solid is then $\{n,
m\}.$

If we take a double cover, branched over the vertices we get a regular
map of type $\{n, 2m\},$ with $v$ vertices, $2f$ faces and 2$e$ edges.
The vertex valency is $2m$ and the face valency is $n$ and so the map has has type
$\{n,2m\}$.  As the $v$ vertices are the branch points we can compute
the genus $g$ from the formula $2g+2=v$, that is
\begin{equation*}
    g=(v-2)/2
\end{equation*}
or just by the Euler formula. 

We note the following result which follows from Proposition~4
of~\cite{jones-00}.  Consider the Riemann surface as a double covering of a
platonic solid with automorphism group $G$, which is branched over the
vertices.  Then the automorphism group of the surface (and of the
lifted map) is $C_2\times G$ if the vertex valency is odd.  We now
consider the Platonic solids in turn.

\begin{enumerate}
    \item \label{platonic1}
    The cube, type $\{4,3\}$, $v=8$, $f=6$, $e=12$, so the double
    cover branched over the vertices has type $\{4,6\}$.  It has $8$
    vertices, $12$ faces and $24$ edges.  The vertices are the
    Weierstrass points so that $2g+2=8$ giving $g=3$, which can also
    be found using the Euler characteristic.  The genus of the surface
    is $3$ and its automorphism group is $C_2\times S_4$.

    The results for other double covers, branched over the vertices is
    similar and we just give the results, for details
    see~\cite{jones-00}.

    \item \label{platonic2}
    The tetrahedron, type $\{3,3\}$ The double cover branched over the
    vertices has type $\{3,6\}$; it has $4$ vertices, $8$ faces and
    has $12$ edges and is of genus $1$.  Its automorphism group is
    $C_2\times A_4$ of order $24$.  It has genus $1$.

    \item \label{platonic3}
    The octahedron, type $\{3,4\}$.  The double cover branched over
    vertices has type $\{3,8\}$; it has 6 vertices, 16 faces and $24$
    edges and genus $2$.  It's automorphism group is $\GL(2,3)$ of
    order $48$ (for the only group of automorphisms of a surface of
    genus $2$ of order $48$ is $\GL(2, 3)$).

    \item \label{platonic4}
    The dodecahedron, type $\{5,3\}$ The double cover branched
    over the vertices has type $\{5,6\}$; it has $20$ vertices, $24$
    faces and $60$ edges It is of genus $9$.  Its automorphism group is
    $C_2\times A_5$.

    \item \label{platonic5}
    The icosahedron, type $\{3, 5\}$ The double-cover branched
    over the vertices has type $\{3,10\}$; It has $12$ vertices and
    $40$ faces and $60$ edges It is of genus $5$.  It's automorphism
    group is $C_2\times A_5$.

    \item \label{platonic6}
    The dihedron, type $\{n, 2\}$.  This solid has $n$ equally spaced
    vertices on the equator, It has $n$ edges and two faces, the upper
    and lower halves of the sphere.  The double cover branched over
    the vertices has type $\{n,4\}$.  It has $n$ vertices, $4$ faces
    and $2n$ edges.  Its genus is $(n-2)/2$.  The automorphism group
    of the Riemann surface that we obtain is twice as large as that of
    $D_n$ and so is equal to $4n=8(g+1)$.  These are the well known
    Accola--Maclachlan surfaces.  (The order of the largest
    automorphism group $M(g)$ of a Riemann surface of genus $g$ obeys
    $M(g)\geq 8(g+1)$; the lower bound is always attained by the
    Accola--Maclachlan surfaces,~\cite{accola-68, maclachlan-69}.)  The
    automorphism group has presentation $\langle r,s \mid
    r^4=s^n=(rs)^2=(r^{-1}s)^2=1\rangle$.

    \item \label{platonic7}
    The hosohedron, type $\{2, n\}$.  This is the dual of the
    dihedron.  It has $2$ vertices, $n$ edges and $n$ faces.  The
    double cover branched over the vertices has type $\{2,2n\}$ and
    has $2$~vertices, $2n$~edges and $2n$~faces and genus~$0$.  Its
    automorphism group is $D_{2n}$ if $n$ is even and $C_2\times D_n$
    if $n$ is odd.

    Branching over face-centres. By considering the dual we see that we
    get the same set of Riemann surfaces as we get by branching over
    faces.

    Branching over edge-centres.  We just find the genera of the
    double covers branched over the edge-centres.  If there are $e$
    edges in the Platonic solid then the double cover branched over
    these edge-centres also has $e$ edge-centres so that $e=2g+2$.  We
    thus find that the double covers branched over the edge-centres
    have genus $g=2$, for the tetrahedron, $g=5$ for the cube or
    octahedron $g=14$, for the icosahedron or dodecahedron.

    \item \label{platonic8}
    Star maps.  A free edge of a map is an edge, which is not a
    loop, that only has one vertex.  The star map $\mathcal {S}_e$ is
    a map on the sphere with one vertex $p$, $1$ face and $e$ free
    edges which all have $p$ as its only vertex.  If we draw the map
    so that the angles between its edges are all equal then Its
    automorphism group is $C_e$ which acts transitively by rotation on
    the edges.  If we consider the double cover branched over the $e$
    edge-centres then its genus is $g=(e-2)/2$ Its automorphism group
    is $C_2\times C_e$ if $e$ is odd and $D_e$ if $e$ is even.  In
    both cases its order is $2e=4g+4$.  Now this is half the order
    that we got when we considered branching over the vertices of of a
    dihedron.  As these edge-centres are the vertices of a regular
    $n$-gon , we get the same set of Riemann surfaces as in (6), the
    Accola--Maclachlan surfaces.
\end{enumerate}

A Fuchsian group interpretation. We briefly mention how this can be
interpreted using Fuchsian groups. First of all we state the
following result which is due to Maclachlan~\cite{maclachlan-69}.

\begin{Th}
    $X=\mathcal U/K, K$ a surface group is hyperelliptic if and only
    if $K$ is a subgroup of index $2$ in a Fuchsian group $\Gamma$ of
    signature $(0; 2^{2g+2})$, where the notation indicates $2g+2$
    periods equal to $2$.  The quotient $\Gamma/K$ acts as the
    hyperelliptic involution.
\end{Th}

We just consider one example, the octahedron with automorphism group
$\GL(2,3)$ as in~(\ref{platonic3}) above.  It lifted to a map of type
$\{3,8\}$  which lies on a Riemann surface $X=\mathcal{U}/K$, $K$ a
surface group.  Let $j$ denote the hyperelliptic involution. By the
results in \cite{jones-78} or \cite{coxeter-80} there is a
homomorphism $\theta:\Gamma\rightarrow \GL(2,3)$ where $\Gamma$ has
presentation $\langle x,y,z \mid x^3=y^8=(xy)^2=1\rangle.$ If
$j$ denotes the hyperelliptic involution (which is represented by
$-I\in \GL(2,3)$), then by standard Fuchsian group
results~\cite{singerman-70}, it
can be shown that $\theta^{-1}(\langle j\rangle)$ is a Fuchsian group
of signature $(0;2^6)$ which contains $K$ with index $2$ so that $X$
is hyperelliptic by Maclachlan's Theorem.

Having found these hyperelliptic surfaces whose automorphism group
acts transitively on the Weierstrass points we now show that these
are all the surfaces with this property.

So suppose that $G<\Aut X$ acts transitively on the Weierstrass points
of a hyperelliptic surface $X$ with hyperelliptic involution~$J$.
Then $X/\langle J\rangle =\Sigma$.  The Weierstrass points on $X$
lower to $2g+2$ points on $\Sigma$, which are permuted transitively by
a finite subgroup of Aut$\Sigma$ which is well-known to be $\PSL(2,
\mathbb{C}).$ The finite subgroups of $\PSL(2, \mathbb{C})$ are the
finite rotation groups of the sphere which are the ones that appear
above in (\ref{platonic1}), \ldots, (\ref{platonic8}).  Thus there are
no more hyperelliptic surfaces with transitive groups of
automorphisms.

We sum up in 

\begin{Th} 
    \label{thrm:7}
    For every integer $g\ge 1$ there is a hyperelliptic surface of
    genus $g$ whose automorphism group acts transitively on the
    Weierstrass points.  These surfaces are the Accola--Maclachlan
    surfaces described in~(\ref{platonic6}).  Otherwise, the only
    hyperelliptic surfaces with transitive automorphism groups have
    genera $1$, $2$, $3$, $5$, $9$ and $14$ which occur
    in the cases~(\ref{platonic1}) to~(\ref{platonic5}) above, and the last from the covering 
    of the icosahedron branched over the edge-centres.
   As all
    these surfaces are covers of regular solids branched over the
    vertices or edge-centres they are unique. (Note: we get two surfaces of genus 5.
    The first is the two-sheeted cover of the icosahedron branched over the 12 vertices, the second being the 
    two-sheeted cover of the cube branched over the edge-centres.)
\end{Th}

\section{Surfaces with automorphism group $\PSL(2, q)$}
\label{sec:surfaces_with_Aut_PSL2q}

If $G$ is a group of automorphisms $X$ of a Riemann surface of genus
$g\ge 2$ then the size of its automorphism group $\Aut X$ is bounded
above by $84(g-1)$.  This is by a classic theorem of Hurwitz.  The
surfaces for which this bound is attained are called Hurwitz surfaces
and their automorphism groups are known as Hurwitz groups.  It is well
known that $G$ is a Hurwitz group if and only if there is a
homomorphism from the $(2,3,7)$ triangle group onto $G$.  The smallest
Hurwitz group is $\PSL(2,7)$ of order 168; the corresponding surfaces
is Klein's surface, which corresponds to the Klein quartic
\begin{equation}
    x^3y+y^3z+z^3x=0. 
\end{equation}

\pagebreak[3]

In \cite{macbeath-69} Macbeath proved 

\nopagebreak

\begin{Th}
    Let $q$ be a prime power.  The group $\PSL(2,q)$ is a Hurwitz
    group if and only if
    \begin{enumerate} 
	\item[(i)] $q=7$ or
	\item[(ii)] $q=p\equiv \pm 1  \pmod 7$ or
	\item[(iii)] $q=p^3,$ where $p\equiv  \pm 2$ or $\pm 3 \pmod 7$.   
    \end{enumerate}
\end{Th}

Thus there are infinitely many surfaces that admit Hurwitz groups as
their automorphism groups, and for this reason it is worth us studying
the transitivity problem for Riemann surfaces whose automorphism group
is $\PSL(2, q)$.  However, there are many other Hurwitz groups.  For
example, Conder \cite{conder-85} showed that the alternating group
$A_n$, is Hurwitz for all $n\ge 168$.  The first $3$ Hurwitz groups are
of the form $\PSL(2, q)$, namely for $q=7, 8, 13$, corresponding to
genera $3$, $7$, $14$.  The next largest is of order $1344$ which acts
on a surface of genus $17$.



\begin{Th}
    \label{thrm:9}
    If $q>15$ then the Hurwitz group $\PSL(2, q)$ does not give a
    transitive action on the Weierstrass points of a Riemann surface
    that it defines.
\end{Th} 

By Theorem~\ref{thrm:5}, the number of fixed
points of an element of order $3$ is greater than $5$, whereas the
number of fixed points of an element of order $2$ is greater than $7$
and hence by Schoeneberg's Theorem these fixed points are Weierstrass
points.  Hurwitz groups are associated to a regular map on a surface,
so the fixed points of elements of order $3$ are vertices of the
associated map, whereas the elements of order 2 are edge-centres.  As
$\Gamma(2,3,7)$ is maximal, the automorphism group of the surface and
the automorphism group of the map coincide.  As no automorphism of the
map can take vertices to edge centres, the automorphism group of the
surface does not act transitively on the Weierstrass points.

The cases $q= 7, 8$  and $13$ are still to be dealt with.   

Firstly, if $q=7$ then the corresponding Riemann surface of genus $3$
is the well-known Klein surface.  Transitivity on the Weierstrass
points was proved in \cite{singerman-97} as the $24$ Weirstrass
points are shown to be the face-centres of the regular map of type
$\{7, 3\}$ on the surface.  It is also proved in \cite{magaard-06} by
more group-theoretic methods.  We just use a counting argument
similar to \cite{magaard-06}. This will also work for $q=8$ and is
applicable to $q=13$ as also noted in \cite{magaard-06}.

\begin{Th}
    \label{thrm:Wpts_trans}
    Let $W$ denote the set of Weierstrass points on the Klein surface
    $X$ of genus $3$.  The automorphism group $\Aut(X)$ acts
    transitively on $W$.
\end{Th}

\begin{proof} 
    As $g=3$, the total weight of the Weierstrass points on $X$ is
    $3^3-3=24$ by Theorem~\ref{thrm:farkas1}.  Now $\PSL(2, 7)$
    besides acting as a group of automorphisms $X$ also acts as a
    group of automorphisms of the regular map of type $\{3, 7\}$
    underlying $X$.  The only points with non-trivial stabilizers are
    the vertices, face-centres and edge-centres, where the stabilizers
    have orders $3$, $7$ and $2$ respectively.  Thus the size of the
    possible orbits are $\sigma_1=168/7=24$, $\sigma_2=168/3=56$,
    $\sigma_3=168/2=84$ and $\sigma_4=168/1=168$, the latter case
    occurring for interior points of the faces of the map, i.e non-geometric points.  If we let
    $w_1$, $w_2$, $w_3$, $w_4$ denote the weights of the Weierstrass
    points in these $4$ orbits we get

    \begin{equation}
	\sum_{i=1}^4w_j\sigma_j=24.
    \end{equation}
    or 
    \begin{equation} 
	24\sigma_1+56\sigma_2+84\sigma_3+168\sigma_4=24.
    \end{equation}
    The only solution is $\sigma_1=1$, $\sigma_2=\sigma_3=\sigma_4=0$.
    Thus only the face-centres are Weierstrass points, and they all
    have weight $1$, and so $\Aut X$ is transitive on the $24$
    Weierstrass points.
\end{proof}

The next largest Hurwitz group is $\PSL(2, 8)$ of order $504$, and
this acts on a surface of genus $7$, often called Macbeath's surface.
(Macbeath wrote an important paper on this Riemann surface,
\cite{macbeath-65}.)

\begin{Th}
    \label{thrm:11}
    Let $W$ denote the set of Weierstrass points of Macbeath's surface
    $X$ of genus $7$.  The automorphism group of $X$ acts transitively
    on $W$.
\end{Th}

\begin{proof}
    The total weight of the Weierstrass points is $7^3-7=336$.  We
    follow the proof of Theorem~\ref{thrm:Wpts_trans}, we now get the
    equation
    \begin{equation}   
	\sum_{i=1}^4w_j\sigma_j=336,
    \end{equation}
    or 
    \begin{equation}
	72w_1+168w_2+252w_3=336. 
    \end{equation}  
    The only solution is $w_2=2$, $w_1=w_3=0$, and so only vertices
    are Weierstrass points of weight $2$ and $\Aut X$ acts
    transitively on the Weierstrass points.
\end{proof}

The next largest Hurwitz group is $\PSL(2, 13)$ of order $1092$.  This
acts as a group of automorphisms of a surface of genus $13$.  The total
weight of the Weierstrass points is $14^3-14=2730$.  Again, we follow
the proof of Theorem~\ref{thrm:Wpts_trans} and we get
\begin{equation*}
    156w_1+364w_2+546w_3+1092w_4=2730
    \end{equation*}

This has  10 solutions in positive integers;  These are 
$(w_1,w_2,w_3,w_4)=$
\begin{enumerate}
\item $(0,0,5,0)$
\item $(0,0,1,2)$
\item $(0,0,3,1)$
\item $(0,3,3,0)$
\item $(0,6,1,0)$
\item $(7,0,3,0)$
\item $(14,0,1,0)$
\item $(0,3,1,1)$
\item $(7,0,1,1)$
\item $(7,3,1,0)$

\end{enumerate}

 If the first case occurs
then $\PSL(2,13)$ acts transitively on the Weierstrass points which
are all edge-centres of weight $5$, but if cases 2--7 occur there are two orbits of Weierstrass points, while if cases 8--10
occur then there are 3 orbits. We cannot decide which of these occur so  
as in~\cite{magaard-06}, we cannot decide transitivity. In fact it is possible that there are one, two, or three orbits 
of Weierstrass points.
(However, because of the large number of solutions of the above equation we might guess
that transitivity does not occur.)
We can now easily deduce

\begin{Th}
    \label{thrm:edge}
 If $X$ is a Hurwitz surface of genus 14, on which $\PSL(2,13)$ acts as a group of automorphisms (there are 3 of them \cite{macbeath-69}) then the edge-centres of the corresponding map of type $\{7,3\}$ are Weierstrass points.
\end{Th}
\begin{proof} By Theorem 5, the number of fixed points of an element of order 2 is 6, so by Schoeneberg's Theorem these are Weierstrass points.  Alternatively, in all the above solutions, $w_3\not= 0.$  
\end{proof}

So even though we cannot decide transitivity we know where 546 of the Weierstrass points are.

{\it Added note. Very recently, Manfred Streit in a detailed investigation of the genus 14 Hurwitz surfaces has shown that the vertices and face-centres of he corresponding map are not Weierstrass points so that $w_1=w_2=0$.  Thus only the first 3 of the above possibilities can occur, so the chances of transitivity have increased from one in ten to one in three! Also, there are at most two orbits of Weierstrass points.}

By Theorems~\ref{thrm:9}, \ref{thrm:Wpts_trans}
and~\ref{thrm:11} we see that  only for the Macbeath-Hurwitz surfaces of genus 14 are we, as yet, unable to decide transitivity.

\medskip

Macbeath \cite{macbeath-69} has shown that for  all $q\not= 9$, that $\PSL(2,q)$ is an image of $\PSL(2,\bf{Z})$, the modular group which is a free product  of $C_2$ and $C_3$.  Thus $\PSL(2,q)$ is an image of the triangle group $(2,3,t)$ for some $t\ge 7$. Let $K_{t,q}$ be the kernel of the homomorphism $\theta:(2,3,t)\rightarrow PSL(2,q)$ and let $X_{t,q}=\mathcal{U}/K_{t,q}$. 
\begin{Th}
\label{thrm:psl}

If $q>15$ or $q=11$ or $q=t=13$ then $\Aut X_{t,q}\cong \PSL(2,q)$ does not act transitively on $X_{t,q}$.
\end{Th}
\begin{proof}
If $q>15$ then the proof follows as in   theorem \ref{thrm:9}. If $q=11$, then by Theorem 5 we see that the number of fixed points of both an element of order $2$ and of order $5$ is equal to $5$, and so by Theorem 3, the edge-centres and face-centres of the corresponding map are Weierstrass points and  so $\Aut X$ does not act transitively. If $q=t=13$, then by Theorem 5 the number of fixed points of the element of order 13 is 6 and so again edge-centres and face-centres are Weierstrass points.
\end{proof}

For $p$ prime let $ \Gamma(p)$ be the principal congruence subgroup of the modular group of level $p$ and let the modular surface $X(p)$be the compactification of $ \mathcal{U}/\Gamma(p)$. Then $X(p)=X_{p,p}$, so by Theorems 10 and 13 we have 
\begin{Co}
 $\Aut X(p)$ acts transitvely on the Weierstrass points of $X(p)$ if and only if $p=7$.
\end{Co}

\section{Platonic surfaces of low genus}
\label{sec:platonic_with_low_genus}

In \cite{singerman-97} the following problem was investigated: When is
a Weierstrass point on a Platonic surface a geometric point?  We
solved this for $g=2,3,4$ and, except for one possible exception, for
$g=5$.  This one exception corresponded to a map of type $\{15,6\}$
where we found that the vertices and face-centres were Weierstrass
points, but we could not decide about the edge centres.  (In
\cite{singerman-97}, $\{m,n\}$, meant vertex valencies $m$ and face
valencies $n$, unlike this paper.)  We found the following results
which we present as in the following two examples.

\begin{center}
    \begin{tabular}{llllllcc}
        \toprule
        & &g& $V$ & $F$ & $E$ & $G$ & order of $G$ 
	\\ 
	\midrule
	(i) & $\{12,3\}$ & 3 & $4^{2}$ & $16^{1}$ & $24$ &
	 $\langle\langle 2,3 \mid 3\rangle\rangle$& $48$
	\\
	(ii) & $\{8,3\}$ & 3&$12^{2}$ & $32$ & $48$ &
	$(2,3,8;3)$ & $96$
	\\
        \bottomrule
    \end{tabular}
\end{center}

In (i) this means that the map type is $\{12,3\}$, the genus is $3$,
there are $4$ faces which are Weierstrass points of weight $2$, there
are $16$ vertices of weight $1$, and the edges are not Weierstrass
points.  As there are Weierstrass points at both vertices and
face-centres (or, alternatively as they have different weights), we do
not get transitivity on the Weierstrass points.

In (ii), we find that all of the Weierstrass points are at the
face-centres, so that , as the automorphism group permutes the faces
transitively, we get a transitive action on the Weierstrass points,
The automorphism group $G$ is usually described in the Coxeter--Moser
notation \cite{coxeter-80}, but if the surface is an Accola--Maclachlan
surface (see section~\ref{sec:hyp_surfaces} (\ref{platonic6}) where we
give the presentation) we denote the group by $AM$; it has order
$8(g+1)$.

We now examine the other cases in \cite{singerman-97}, where all the
Weierstrass points are of the same type.  We have to be careful of the
fact that it is possible for a Riemann surface to carry more than one
map (see~\cite{singerman-03}).  This is when there is an inclusion
relation between triangle groups, for example $(2,n,n)<(2,4,n)$, so
that when a surface carries a regular map of type $(n,n)$ it may also
carries one of type $(4,n)$.  Also $(2,n,2n)<(2,3,2n)$, (see
\cite{singerman-72}).  We then choose the larger group.

We find the following cases:
\nopagebreak
\begin{center}
    \begin{tabular}{llllllcc}
        \toprule
        & &g& $V$ & $F$ & $E$ & $G$ & order of $G$ 
	\\ 
	\midrule
	(1)&$\{6,4\}$&2&$6^1$&$4$&$12$&$AM$&$24$
	\\
	(2)&$\{8,3\}$&2&$16$&$6^1$&$24$&$\GL(2,3)$&$48$
	\\
	(3)&$\{8,4\}$&3&$8^3$&$4$&$16$&$AM$&$32$
	\\
	(4)&$\{6,4\}$&$3$&$12$&$8^3$&$24$&$S_4\times C_2$&$48$
	\\
	(5)&$\{8,3\}$&$3$&$32$&$12^2$&$48$&$(2,3,8;3)$&$96$
	\\
	(6) & $\{7,3\}$ & $3$&$56$ & $24^1$ & $84$ &
	$\PSL(2,7)$ & $168$
	\\
	(7)&$\{10,4\}$&$4$&$10^6$&$4$&$20$&$AM$&$40$
	\\
	(8)&$\{5,4\}$&$4$&$30$&$24$&$60^1$&$S_5$&$120$
	\\
	(9)&$\{12,4\}$&$5$&$12^{10}$&$4$&$24$&$AM$&$48$
	\\
	(10)&$\{3,10\}$&$5$&$40$&$12^{10}$&$60$&$C_2\times A_5$&$120$
	\\
	(11)&\{8,3\}&$5$&$64$&$24^5$&$96$&$\SL(2, {\bf Z/}8)$&$192$
	\\
	(12)&$\{5,4\}$&$5$&$40^3$&$32$&$80$&***&$160$
	\\
        \bottomrule
\end{tabular}
\end{center}

Now (1), (2), (3), (4), (7), (9) and (10) correspond to hyperelliptic
surfaces and so have already appeared in
section~\ref{sec:hyp_surfaces}.  Also (6) is Klein's surface which appeared in
section~\ref{sec:surfaces_with_Aut_PSL2q}.  The cases (5), (8),
(11) and (12) give new examples where transitivity occurs. ((5) is the Fermat curve $F_4$, (see section 10),  (8) is Bring's surface, (11) is a double cover of the Fermat curve, sometimes known as Wiman's curve  and (12) is known as Humbert's curve.  
For these curves see \cite{keem-10}. In (12) the group *** has
presentation $\langle{r,s \mid
r^5=s^4=(rs)^2=(rs^{-1})^4=1}\rangle$.  For more details on the
regular map and the group see~\cite{coxeter-39}*{p.~135}.
	




\section{Simple Weierstrass points}
 
A Weierstrass point of weight $1$ is called a simple Weierstrass
point.  By equation~\ref{eq:2}, if all Weierstrass points of a surface of genus
$g$ have weight $1$ then the number of Weierstrass points is equal to
$g^3-g$.  As shown in \cite{rauch-59} in some sense this is the
general case.  However the following result shows that this is certainly not
the case for Riemann surfaces whose automorphism group is transitive
on the Weierstrass points.   
\begin{Th}
    If there is a transitive action on the Weierstrass points on a
    Riemann surface $X$ of genus $g>2$ and the Weierstrass points are
    simple then either $g=4$ and $X$ is Bring's surface, or $g=3$ and $X$ is Klein's surface or $g=3$ and $\Aut X\cong S_4$.  If the latter case occurs then  $\Aut X$ acts freely on the Weierstrass points.  \end{Th}
 
\begin{proof}  
    Let $|W|$ denote the number of Weierstrass points on $X$. As these
    Weierstrass points form an orbit under the action of $G=\Aut X$
    then $|W||G_p|=|G|\le 84(g-1)$, where $G_p$ is the stabilizer of a Weierstrass point 
    $p$ so that $|W|\le 84(g-1)$.  By Theorem~\ref{thrm:farkas2}, $|W|
    =g^3-g$.  Hence $g^3-g\le 84(g-1)$and so $g\le 8$.  In general,
    transitivity implies that $|W|\leq M(g)$, the order of the largest
    group of automorphism for a surface of genus $ g$.  We now note that if $X=\mathcal U/K$  ($K$ a surface group)  has large enough automorphism group then $X$ must be platonic.  For $\Aut X\cong \Gamma/K$ and if $|\Aut X|>24(g-1)$, then by the Riemann-Hurwitz formula, the measure of a fundamental region for $\Gamma$ is < 1/12 and so $\Gamma$ must be a $(2,m,n)$ triangle group so that $X$ is platonic.
So the largest automorphism groups are automorphism groups of regular maps, 
 so we can consult \cite{conder-01} where we find that $M(6)=150, M(7)=504$, (for the 
    Macbeath surface) and $M(8)=336.$  However, $6^3-6=210>150$,
    and $8^3-8=504>336$ so we do not get transitivity for $g=6$ or
    $8$.  If $g=7$ then the largest automorphism group is $504$ which
    gives Macbeath's surface., the unique surface of genus 7 with  504 automorphisms, \cite {macbeath-69}. But we saw in the proof of
    Theorem~\ref{thrm:11} that the weights of the Weierstrass points
    of Macbeath's surface are $2$ and so are not simple.  The next
    highest order for an automorphism group in genus 7 is $288$ corresponding to
    a $(2,3,8)$ group.  As $288<7^3-7$ we cannot have transitivity.     

Thus the genus of $X$ is $2$, $3$, $4$ or $5$.  As $X$ is platonic we can use the lists in \cite{singerman-97} to get information about the Weierstrass points.

\begin{enumerate} 

\item If $g=5$  then $|G|=(5^3-5)|G_p|=120|G_p|$. The surface $X$ must then be platonic.  From \cite{singerman-97} we find that $|G|=120$ and the map has type $\{3,10\}$ and the Weierstrass points have weight 10. (In fact this surface is the two-sheeted covering of the icosahedron branched over the vertices that we found as example (5) in section 4.) 

\item If $g=4$ then $|G|=60|G_p|$.  From \cite{singerman-97} we find that either $|G|=120$ or $|G|=60$, ($G$ being $S_5$ or $A_5$.) In the first case the map has type $\{4,5\}$ and is the map underlying Bring's surface. In the second case we get a map of type $\{5,5\}$ on the the same surface. This is because of the inclusion relationship $(2,5,5)<(2,4,5)$, and so the first map is a truncation of the second map. (see \cite{singerman-03}.  By \cite{singerman-97} we see that the 60 edge centres are simple Weierstrass points.

\item
If $g=3$ then $|G|=24|G_p|$. If the surface is platonic then again we consult \cite{singerman-97} and we find that we must have $|G|=168$, the map has type $\{3,7\}$ so the surface is Klein's surface. However, we cannot guarantee that $X$ must be platonic. A possibility is that $|G|=24$, $G\cong S_4$  and $G$ lifts up to a Fuchsian group of signature $\{2,2,2,3\}$.  If this were to occur then $G$ would permute the Weierstrass points freely. Also, as the complex dimension of the Teichm\"uller space of $\{0,2,2,2,3\}$ groups is equal to 1 there will be an uncountable number of surfaces of genus 3 admitting $S_4$ as a group of automorphisms.  It is unclear whether there could be any where the group does not fix any Weierstrass points or whether the weight of the Weierstrass points is equal to one.
\end{enumerate}
\end{proof}

\section{A necessary condition for transitivity}
\label{sec:necessary_condition} 

Let $G$ be a group of automorphisms of a Riemann surface $X$. Then
$G$ lifts to a Fuchsian group $\Gamma$ with periods $m_1,\ldots
,m_r$, corresponding to elliptic elements $x_i$  ($i=1,\ldots,r$)
with fixed points $a_i$  ($i=1,\ldots,r$). Furthermore, there is a
homomorphism $\phi:\Gamma\rightarrow G$ with torsion-free kernel.
All points of $\mathcal U$ with non-trivial stabilizer have the form
$\gamma a_i$,  ($\gamma\in \Gamma$) this point being stabilized by
$\gamma x_i\gamma^{-1}$ of order $m_1$.  The point $\gamma a_i$
projects to a point in $X$  stabilized by a cyclic group of order
$m_i$, and all fixed points have this form. Now suppose that $G$ acts
transitively on the the Weierstrass points $W$ of $X$. Then $W$ is an
orbit so that the 60 edge-centres are simple Weierstrass points.
\begin{equation}
    |W||G_p|=|G|
\end{equation}
where $G_p$ is the stabilizer of $p$, so that $|G_p|=m_i$, one of the
periods of $\Gamma$.

As all points of $W$ have the same weight $w$, it follows from
Theorem~\ref{thrm:farkas1}, that 
\begin{equation}
    |W|w=g^3-g
\end{equation}

We can therefore deduce 

\begin{Th}
    If $G$ acts transitively on the Weierstrass points then 
    \begin{equation}
	w=\frac{|G_p|(g^3-g)}{|G|}.
    \end{equation}
\end{Th}

(Thus the weight of a Weierstrass point is calculated in group
theoretic terms!  However we do have the group theoretic condition of
transitivity.)

\textbf{Example}.  If $G$ is a Hurwitz group of order $84(g-1)$,  a
homomorphic image of a triangle group $(2,3,7),$ then $m_i$ is one of
$2$, $3$ or $7$, so that if $G$ acts transitively on the Weierstrass
points then $g(g+1)$ is divisible by $12$, $28$ or $42$.

For example, the Hurwitz group of order $1344$ which acts on a
surface of genus $17$, does not act transitively.  Unfortunately,
this does not tell us anything about $\PSL(2, 13)$ acting on a
surface of genus $14$.

\section{bi-elliptic surfaces}

A Riemann surface $X$ is called \emph{$\gamma$-elliptic } if $X$
admits an involution~$J$ such that $X/\langle J \rangle$ has genus
$\gamma$.  $0$-hyperelliptic is just hyperelliptic, and
$1$-hyperelliptic is called bi-elliptic.  To investigate
transitivity on bi-elliptic surfaces we need the following
theorems of Kato~\cite{kato-79} and Garcia~\cite{garcia-86}.

\begin{Th}[Kato]
    Let $X$ be a non-hyperelliptic Riemann surface of genus $g\ge
    3$. Then for every point $p$ of $X$,
    \begin{equation*}
	0\le w_p\le 
	\begin{cases}
	    \displaystyle \frac {g(g-1)}{3} 
	    & 
	    \text{if $g=3,4,6,7,9,10$} \\[2.5ex]
	    \displaystyle \frac{g^2-5g+10}{2} 
	    & 
	    \text{otherwise}
	\end{cases}
    \end{equation*}
    where $w_p$ is the weight of $p$.
\end{Th}
 
Kato also exhibited a link between this bound and
bi-elliptic surfaces.
  
\begin{Th}
    Let $X$ be a Riemann surface of genus $g\ge 11$.  Then~$X$ is
   bi-elliptic if and only if there exists $p\in X$ with
    weight $w_p$ such that
    \begin{equation*}
	\frac{g^2-5g+6}{2}\le w_p< \frac{g^2-g}{2}.
    \end{equation*}
    In this case the possible values of $w_p$ are $\frac{g^2-5g+6}{2}$
    or $\frac{g^2-5g+10}{2}$.
\end{Th}

Garcia~\cite{garcia-86} proved that both these bounds are attained. 

\begin{Th}
    Let $X$ be an bi-elliptic Riemann surface of genus
    $g\ge 11$, $g\not =15$.  If there exists $p\in X$ such that
    $w_p=\frac{g^2-5g+10}{2}$ or $w_p=\frac{g^2-5g+6}{2}$, then $\Aut
    X$ does not act transitively on the Weierstrass points $W$ of $X$.
\end{Th}

\begin{proof}
    We prove this by contradiction so suppose that $\Aut X$ does act
    transitively on $W$.  If $w_p=\frac{g^2-5g+10}{2}$ then all
    Weierstrass points of $X$ have this weight.  Thus
    $\frac{g^2-5g+10}{2}$ divides $g^3-g$ the quotient being $|W|$.
    Thus
    \begin{equation*}
	|W|=2g+10+\frac{28g-100}{g^2-5g+10}.
    \end{equation*}
    Therefore $\nu(g)=\frac{28g-100}{g^2-5g+10}$ must be an integer.
    However, $\nu(g)$ decreases for $g>11$ and for $g=11$, $\nu(g)<3$
    and so $\nu(g)=1$ or $2$.  There is no integral solution if $g=1$
    and $\nu(g)=2$ implies that $g=4$ or $g=15$.  Thus for $g>11$,
    $g\neq 15$ we cannot get a transitive action.

    Similarly $\frac{g^2-5g+6}{2}$ cannot divide $g^3-g$ for $g\ge
    11$.
\end{proof}

We thus have

\begin{Th}
    Let $X$ be an bi-ellipticsurface of genus $g>11$,
    $g\neq 15$.  Then $\Aut X$ does not act transitively on the
    Weierstrass points of $X$.
\end{Th}

Using similar results of Garcia~\cite{garcia-86} we can also show that
with the possible exception of a small number of values of $g$, that
no $2$-hyperelliptic surface of genus $g$ can have an automorphism
group that acts transitively on the Weierstrass points.






\section{Fermat surfaces}

The Fermat surface $\mathbf{F}_n$ is the
Riemann surface of the projective algebraic curve 
\begin{equation}
\{x,y,z) \mid x^n+y^n+z^n=0\}
\end{equation}

Let $\Gamma(n,n,n)$ be the triangle group with the presentation 
{$\langle x,y,z|x^n=y^n=z^n=xyz=1\rangle$.}
 If we abelianize this group we get ${\bf Z}_n\oplus{ \bf Z}_n$ so that the commutator subgroup $K_n$ has index $n^2$. It is known that $\bf{F}_n \cong {\mathcal U}/K_n,$  [10]. We then find that $K_n$ has genus $(n-1)(n-2)/2$. As $K_n$ is characteristic in $\Gamma(n,n,n)$ and as $\Gamma(n,n,n)$ is normal in the maximal triangle group $\Gamma(2,3,2n)$ with $\Gamma(2,3,2n)/\Gamma(n,n,n)\cong S_3$, it follows that $Aut F_n\cong ({\bf Z}_n\oplus{ \bf Z}_n)\rtimes S_3$

We now describe how $(\Z_n\oplus \Z_n)\rtimes S_3$ acts on the points
of the Fermat curve $\F_n$.  We first define an action of $\Z_n^3$ on
$\F_n$.  If $(a,b,c)\in \Z_n^3$ and $(x,y,z)\in \F_n$ we let
$(a,b,c)\wedge (x,y,z)=(e^{2\pi i(a/n)}, e^{2\pi i(b/n)}, e^{2\pi
i(c/n)})$.  The kernel of this action consists of the set of points
$(a,b,c)$ where $a=b=c$ which form a group isomorphic to $\Z_n$ and so
if we factor out this subgroup we get an action of $\Z_n\oplus \Z_n$ as
a group of automorphisms of $\F_n$.  Also there is an action of $S_3$
by permuting the coordinates and so we find an action of $\F_n \cong
(\Z_n\oplus \Z_n)\rtimes S_3$ as a group of automorphisms of $\F_n$.

\subsection*{Weierstrass points on the Fermat curve.}  

Let $\alpha= e^{2\pi i/n}$ and $\beta=e^{\pi i/n}$. Then the $3n$
points of $\F_n$ of the form 
\begin{equation*}
    (0,\alpha^j, \beta), (\beta, 0, \alpha^j), (\alpha^j, \beta, 0)
\end{equation*}
($j=0,\ldots, n-1)$, are called the trivial points of $\bf{F_n}$ and
Hasse~\cite{hasse-75} showed that these points are Weierstrass points
of weight
\begin{equation*}
    (n-1)(n-2)(n-3)(n+4)/24.
\end{equation*}
When $n=4$ this is equal to $2$ and so the total weight is $24$.
By~\eqref{eq:farkas} in section~\ref{sec:prelim}, this
shows that these are all the Weierstrass points of~$\F_4$.  For $n\ge
5$ there are more Weierstrass points, called the Leopoldt
points~\cite{rohrlich-82}.  These are the $3n^2$ points of the form
$(\gamma,\beta_1,\beta_2)$, where $\gamma^n=2,
\beta_1^n=\beta_2^n=-1$.  According to Towse \cite{towse-97}, their
weight is $\ge (n-1)(n-3)/8$ if $n$ is odd and $(n-2)(n-4)/8$ if $n$
is even and this inequality is an equality if $n\le 8$.

When $n=5$, the total weight of the Leopoldt points is equal to~$1$
and so the total weights of the trivial points and the Leopoldt points
is equal to $15\cdot 9 + 75\cdot 1=210=6^3-6$ and so by
Theorem~\ref{thrm:farkas2}, the trivial points and the Leopoldt points
include all of the Weierstrass points.  When $n\ge 6$, a similar count
shows that there must be further Weierstrass points.

\begin{Th}
    The automorphism group of $\F_n$ acts transitively on the
    Weierstrass points if and only if $n=4$.
\end{Th}

\begin{proof}  
    We consider the action of the automorphism group of the Fermat
    curve defined above.  Clearly the set of trivial points is an
    orbit.  The set of Leopoldt points is also an orbit and these
    orbits are distinct.  For $n=4$ the only Weierstrass points are
    the trivial ones and thus we get transitivity.  For $n>4$, there
    are Weierstrass points of Leopoldt type and so we do not get
    transitivity on all the Weierstrass points.
\end{proof}

\begin{notes}\mbox{}\par
    \begin{enumerate}
        \item As Rohrlich remarked~\cite{rohrlich-82}, the automorphisms like
        \begin{equation*}
        (x,y,z)\mapsto (\alpha x, y,z)
        \end{equation*}
	(or multiply the 2nd or 3rd coordinates by $\alpha$) are  automorphisms that fix  $n$ trivial points
	 and hence for $n>4$ we know these
	are Weierstrass points by Schoeneberg's Theorem.  Also, we can
	see that the automorphim $(x,y,z)\mapsto (x,z,y)$ fixes the
	Leopoldt points of the form $(\gamma,\beta, \beta)$ showing
	these are Weierstrass points.  As all Leopoldt points form an
	orbit under $\Aut \F_n$ we see that they are all Weierstrass
	points.
        
        \item Because of the inclusion relationship
        $\Gamma(4,4,4)<\Gamma(2,3,8)$ we can deduce that $\F_4$ is the
        Platonic surface of type $\{8,3\}$ in
        section~\ref{sec:platonic_with_low_genus}.
    \end{enumerate}
\end{notes}

\section{Conclusion} 

We showed in Theorem~\ref{thrm:7} that there is one family of
hyperelliptic surfaces, namely the Accola--Maclachlan surfaces, where
the automorphism group acts transitively on the Weierstrass points.
However for the families of non-hyperelliptic surfaces we have
considered there seem to be few examples and the surfaces we have
found which have this property such as the Klein surface, Bring's
surface, the Macbeath surface, the Fermat curve $\F_4$ are important for other reasons.  It
would be interesting to search for other non-hyperelliptic Riemann surfaces whose
automorphism group acts transitively on the Weierstrass points.

%
%


\begin{bibdiv}
\begin{biblist}

\bib{accola-68}{article}{
   author={Accola, R. D. M.},
   title={On the number of automorphisms of a closed Riemann surface},
   journal={Trans. Amer. Math. Soc.},
   volume={131},
   date={1968},
   pages={398--408},
   issn={0002-9947},
   review={\MR{0222281 (36 \#5333)}},
}


\bib{conder-85}{article}{
   author={Conder, M. D. E.},
   title={The symmetric genus of alternating and symmetric groups},
   journal={J. Combin. Theory Ser. B},
   volume={39},
   date={1985},
   number={2},
   pages={179--186},
   issn={0095-8956},
   review={\MR{811121 (87e:20072)}},
}


\bib{conder-01}{article}{
   author={Conder, M.},
   author={Dobcs{\'a}nyi, Peter},
   title={Determination of all regular maps of small genus},
   journal={J. Combin. Theory Ser. B},
   volume={81},
   date={2001},
   number={2},
   pages={224--242},
   issn={0095-8956},
   review={\MR{1814906 (2002f:05088)}},
}

\bib{coxeter-39}{article}{
   author={Coxeter, H. S. M.},
   title={The abstract groups $G\sp {m,n,p}$},
   journal={Trans. Amer. Math. Soc.},
   volume={45},
   date={1939},
   number={1},
   pages={73--150},
   issn={0002-9947},
   review={\MR{1501984}},
}

\bib{coxeter-80}{book}{
   author={Coxeter, H. S. M.},
   author={Moser, W. O. J.},
   title={Generators and relations for discrete groups},
   series={Ergebnisse der Mathematik und ihrer Grenzgebiete [Results
in
   Mathematics and Related Areas]},
   volume={14},
   edition={4},
   publisher={Springer-Verlag},
   place={Berlin},
   date={1980},
   pages={ix+169},
   isbn={3-540-09212-9},
   review={\MR{562913 (81a:20001)}},
}

\bib{farkas-80}{book}{
  author={Farkas, H. M.},
  author={Kra, I.},
  title={Riemann surfaces},
  series={Graduate Texts in Mathematics},
  volume={71},
  publisher={Springer-Verlag},
  place={New York},
  date={1980},
  pages={xi+337},
  isbn={0-387-90465-4},
   review={\MR{583745 (82c:30067)}},
}

\bib{garcia-86}{article}{
   author={Garc{\'{\i}}a, A.},
   title={Weights of Weierstrass points in double coverings of curves
of
   genus one or two},
   journal={Manuscripta Math.},
   volume={55},
   date={1986},
   number={3-4},
   pages={419--432},
   issn={0025-2611},
   review={\MR{836874 (87f:14016)}},
}

\bib{hasse-75}{book}{
   author={Hasse, H.},
   title={Mathematische Abhandlungen. Band 2},
   language={German},
   note={Herausgegeben von Heinrich Wolfgang Leopoldt und Peter
Roquette},
   publisher={Walter de Gruyter, Berlin-New York},
   date={1975},
   pages={xv+525},
   review={\MR{0465757 (57 \#5648b)}},
}

\bib{jones-78}{article}{
   author={Jones, G. A.},
   author={Singerman, D.},
   title={Theory of maps on orientable surfaces},
   journal={Proc. London Math. Soc. (3)},
   volume={37},
   date={1978},
   number={2},
   pages={273--307},
   issn={0024-6115},
  review={\MR{0505721 (58 \#21744)}}  
}

\bib{jones-96}{article}{
   author={Jones, G. A.},
   author={Singerman, D.},
   title={Bely\u\i\ functions, hypermaps and Galois groups},
   journal={Bull. London Math. Soc.},
   volume={28},
   date={1996},
   number={6},
   pages={561--590},
   issn={0024-6093},
   review={\MR{0505721 (58 \#21744)}}  
}

\bib{jones-00}{article}{
  author={Jones, G. A.},
  author={Surowski, D. B.},
  title={Regular cyclic coverings of the Platonic maps},
  journal={European J. Combin.},
  volume={21},
  date={2000},
  number={3},
  pages={333--345},
  issn={0195-6698},
  review={\MR{1750888 (2001a:05076)}},
}

\bib{kato-79}{article}{
   author={Kato, T.},
   title={Non-hyperelliptic Weierstrass points of maximal weight},
   journal={Math. Ann.},
   volume={239},
   date={1979},
   number={2},
   pages={141--147},
   issn={0025-5831},
   review={\MR{519010 (80b:30040)}},
}

\bib{keem-10}{article}{
	author={Keem, C.},
	author={Martens, G.},
	title={On curves with all Weierstrass points of maximal weight},
	journal={Arch. Math.},
	volume={94},
	date={2010},
	pages={339--349},
}

\bib{lewittes-63}{article}{
   author={Lewittes, Joseph},
   title={Automorphisms of compact Riemann surfaces},
   journal={Amer. J. Math.},
   volume={85},
   date={1963},
   pages={734--752},
   issn={0002-9327},
   review={\MR{0160893 (28 \#4102)}},
}

		


\bib{macbeath-65}{article}{
   author={Macbeath, A. M.},
   title={On a curve of genus $7$},
   journal={Proc. London Math. Soc. (3)},
   volume={15},
   date={1965},
   pages={527--542},
   issn={0024-6115},
   review={\MR{0177342 (31 \#1605)}},   
}

\bib{macbeath-69}{article}{
   author={Macbeath, A. M.},
   title={Generators of the linear fractional groups},
   conference={
      title={Number Theory},
      address={Proc. Sympos. Pure Math., Vol. XII, Houston, Tex.},
      date={1969},
   },
   book={
      publisher={Amer. Math. Soc.},
      place={Providence, R.I.},
   },
   date={1969},
   pages={14--32},
   review={\MR{0262379 (41 \#6987)}},
}

\bib{macbeath-73}{article}{
   author={Macbeath, A. M.},
   title={Action of automorphisms of a compact Riemann surface on the
first
   homology group},
   journal={Bull. London Math. Soc.},
   volume={5},
   date={1973},
   pages={103--108},
   issn={0024-6093},
   review={\MR{0320301 (47 \#8840)}},
}

\bib{maclachlan-69}{article}{
   author={Maclachlan, C.},
   title={A bound for the number of automorphisms of a compact Riemann
   surface. },
   journal={J. London Math. Soc.},
   volume={44},
   date={1969},
   pages={265--272},
   issn={0024-6107},
   review={\MR{0236378 (38 \#4674)}},
}

\bib{magaard-06}{article}{
  author={Magaard, K.},
  author={V{\"o}lklein, H.},
  title={On Weierstrass points of Hurwitz curves},
  journal={J. Algebra},
  volume={300},
  date={2006},
  number={2},
  pages={647--654},
  issn={0021-8693},
     review={\MR{2228214 (2007a:14036)}},	
}

\bib{rauch-59}{article}{
   author={Rauch, H. E.},
   title={Weierstrass points, branch points, and moduli of Riemann
surfaces},
   journal={Comm. Pure Appl. Math.},
   volume={12},
   date={1959},
   pages={543--560},
   issn={0010-3640},
   review={\MR{0110798 (22 \#1666)}},
}

\bib{rohrlich-82}{article}{
   author={Rohrlich, D. E.},
   title={Some remarks on Weierstrass points},
   conference={
      title={Number theory related to Fermat's last theorem
(Cambridge,
      Mass., 1981)},
   },
   book={
      series={Progr. Math.},
      volume={26},
      publisher={Birkh\"auser Boston},
      place={Mass.},
}   date={1982},
   pages={71--78},
   review={\MR{685289 (84d:14008)}},
   }
   \bib{schoeneberg-51}{article}{
   author={Schoeneberg, Bruno},
   title={\"Uber die Weierstrass-Punkte in den K\"orpern der elliptischen
   Modulfunktionen},
   language={German},
   journal={Abh. Math. Sem. Univ. Hamburg},
   volume={17},
   date={1951},
   pages={104--111},
   issn={0025-5858},
   review={\MR{0044568 (13,439c)}},
}

   



\bib{singerman-70}{article}{
   author={Singerman, D.},
   title={Subgroups of Fuschian groups and finite permutation groups},
   journal={Bull. London Math. Soc.},
   volume={2},
   date={1970},
   pages={319--323},
   issn={0024-6093},
   review={\MR{0281805 (43 \#7519)}},
}



\bib{singerman-72}{article}{
   author={Singerman, D.},
   title={Finitely maximal Fuchsian groups},
   journal={J. London Math. Soc. (2)},
   volume={6},
   date={1972},
   pages={29--38},
   issn={0024-6107},
   review={\MR{0322165 (48 \#529)}},
}


\bib{singerman-97}{article}{
   author={Singerman, D.},
   author={Watson, P. D.},
   title={Weierstrass points on regular maps},
   journal={Geom. Dedicata},
   volume={66},
   date={1997},
   number={1},
   pages={69--88},
   issn={0046-5755},
   review={\MR{1454933 (98d:30055)}},
}

\bib{singerman-03}{article}{
   author={Singerman, D.},
   author={Syddall, R. I.},
   title={The Riemann surface of a uniform dessin},
   journal={Beitr\"age Algebra Geom.},
   volume={44},
   date={2003},
   number={2},
   pages={413--430},
   issn={0138-4821},
   review={\MR{2017042 (2004k:14053)}},
}

\bib{towse-97}{article}{
   author={Towse, C.},
   title={Weierstrass weights of fixed points of an involution},
   journal={Math. Proc. Cambridge Philos. Soc.},
   volume={122},
   date={1997},
   number={3},
   pages={385--392},
   issn={0305-0041},
   review={\MR{1466643 (98i:14033)}},
}


\bib{weber-05}{article}{
   author={Weber, M.},
   title={Kepler's small stellated dodecahedron as a Riemann surface},
   journal={Pacific J. Math.},
   volume={220},
   date={2005},
   number={1},
   pages={167--182},
   issn={0030-8730},
   review={\MR{2195068 (2006j:30075)}},
}
\end{biblist}
\end{bibdiv}
\end{document}